\documentclass[12pt]{article}
\usepackage{fullpage, amsmath, amssymb, amsfonts, amsthm, stmaryrd, url, hyperref}

\newtheorem{theorem}{Theorem}[section]
\newtheorem{lemma}[theorem]{Lemma}

\newtheorem{prop}[theorem]{Proposition}
\theoremstyle{definition}
\newtheorem{defn}[theorem]{Definition}
\newtheorem{hypothesis}[theorem]{Hypothesis}
\newtheorem{example}[theorem]{Example}
\newtheorem{remark}[theorem]{Remark}
\newtheorem{question}[theorem]{Question}

\numberwithin{equation}{theorem}

\newcommand{\FF}{\mathbb{F}}
\newcommand{\Fp}{\mathbb{F}_p}

\newcommand{\RR}{\mathbb{R}}
\newcommand{\ZZ}{\mathbb{Z}}
\newcommand{\Zp}{\mathbb{Z}_p}
\newcommand{\calC}{\mathcal{C}}

\newcommand{\frako}{\mathfrak{o}}

\DeclareMathOperator{\an}{an}
\DeclareMathOperator{\Cont}{Cont}
\DeclareMathOperator{\cont}{cont}

\DeclareMathOperator{\Map}{Map}

\DeclareMathOperator{\SL}{SL}
\DeclareMathOperator{\SO}{SO}
\DeclareMathOperator{\Sp}{Sp}

\title{The Hochschild-Serre property for some $p$-adic analytic group actions}
\author{Kiran S. Kedlaya}
\date{June 11, 2015}

\begin{document}

\maketitle

\begin{abstract}
Let $H \subseteq G$ be an inclusion of $p$-adic Lie groups. When $H$ is normal or even subnormal in $G$, the Hochschild-Serre spectral sequence implies that any continuous $G$-module whose $H$-cohomology vanishes in all degrees also has vanishing $G$-cohomology. With an eye towards applications in $p$-adic Hodge theory, we 
extend this to some cases where $H$ is not subnormal, assuming that the $G$-action is analytic in the sense of Lazard.
\end{abstract}

\section{Introduction}

Let $H \subseteq G$ be an inclusion of groups and let $M$ be a $G$-module. If $H$ is normal, then the \emph{Hochschild-Serre spectral sequence} \cite{hochschild-serre}
has the form
\begin{equation} \label{eq:hs spectral sequence}
E_2^{p,q} = H^p(G/H, H^q(H, M)) \Longrightarrow H^{p+q}(G,M).
\end{equation}
(This is sometimes also called the \emph{Lyndon spectral sequence} in recognition of a similar prior result \cite{lyndon} which did not explicitly exhibit the spectral sequence.)
If $H$ is not normal, one can still ask to what extent the $G$-cohomology of $M$ is determined by the $H$-cohomology. In particular, one can ask whether 
for any morphism $M \to N$ of $G$-modules such that $H^i(H,M) \to H^i(H,N)$ is an isomorphism for all $i \geq 0$, $H^i(G,M) \to H^i(G,N)$ is also an isomorphism for all $i \geq 0$; in this case, we say that the inclusion $H \subseteq G$ of groups has the \emph{HS (Hochschild-Serre) property}.
Thanks to \eqref{eq:hs spectral sequence}, the HS property holds  when $H$ is \emph{subnormal} in $G$, i.e., there exists a finite sequence $H = H_0 \subset H_1 \subset \cdots \subset H_m = G$ in which each inclusion $H_i \subset H_{i+1}$ is normal. 
On the other hand, it is not difficult to produce examples of inclusions of finite groups for which the HS property fails; see for instance Example~\ref{exa:no HS}.

One can also exhibit an analogue of the Hochschild-Serre spectral sequence for normal inclusions of topological groups, which again implies the HS property for subnormal inclusions; see \cite{flach}. The main result of this paper (Theorem~\ref{T:kill analytic cohomology}) is a restricted analogue of the HS property for certain non-subnormal inclusions of $p$-adic Lie groups, which applies only to the category of topological modules which are of characteristic $p$ and \emph{analytic} in the sense of Lazard \cite{lazard-groups}. It is crucial that the cohomology groups of such modules can be computed using either continuous or analytic cochains; this makes it possible to quantify the statement that an analytic group action of a $p$-adic Lie group is ``approximately abelian.''

We illustrate this theorem with some examples which arise from $p$-adic Hodge theory.
To be precise, these examples come from upcoming joint work with 
Liu \cite{kedlaya-liu2} on generalizations of the Cherbonnier-Colmez theorem on descent of $(\varphi, \Gamma)$-modules \cite{cherbonnier-colmez}, in the style of our new approach to the original theorem of Cherbonnier and Colmez \cite{kedlaya-new-phigamma}.

\section{The HS property for discrete groups}

For context, we begin with some remarks on the HS property for discrete groups.

\begin{defn}
For $G$ a group and $M$ a $G$-module, we say that $M$ has \emph{totally trivial $G$-cohomology} if $H^i(G,M) = 0$ for all $i \geq 0$.
Note that for given $G,H,\calC$, the HS property can be formulated as the statement that
any $M \in \calC$ with totally trivial $H$-cohomology also has totally trivial $G$-cohomology.
\end{defn}

\begin{remark} \label{R:no triviality descent}
If $H \subset G$ is a proper inclusion of groups and $M$ is a $G$-module with totally trivial $G$-cohomology, $M$ need not have totally trivial $H$-cohomology.
\end{remark}

\begin{prop} \label{P:totally trivial}
Let $G$ be a finite $p$-group and let $M$ be a $G$-module. The following conditions are equivalent.
\begin{enumerate}
\item[(a)]
The $G$-module $M$ has totally trivial $G$-cohomology.
\item[(b)]
The group $M$ is uniquely $p$-divisible (i.e., is a module over $\ZZ[p^{-1}]$)
and $H^0(G,M) = 0$.
\end{enumerate}
\end{prop}
\begin{proof}
For $i>0$, $H^i(G,M)$ is a torsion group killed by the order of $G$ \cite[\S 2.4, Proposition~9]{serre-cohomology}; hence (b) implies (a). Conversely, the $p$-torsion subgroup $M[p]$ of $M$ has the property that $H^0(G, M[p]) = M[p]$ injects into $H^0(G,M)$. Consequently, if $M$ has totally trivial $G$-cohomology, then on one hand multiplication by $p$ is injective on $M$; on the other hand, the same is then true for $pM$ (which is isomorphic to $M$ as a $G$-module) and $M/pM$ (by the long exact sequence in cohomology), but the latter forces $M/pM = 0$. Hence (a) implies (b).
\end{proof}

\begin{remark} \label{R:p-group}
Proposition~\ref{P:totally trivial} implies the HS property for inclusions of finite $p$-groups, although this is already clear because such inclusions are always subnormal.
An immediate corollary is that if $H$ is a subgroup of a normal subgroup $P$ of $G$ which is a finite $p$-group, then $H \subseteq G$ has the HS property.
\end{remark}

\begin{example}[Serre] \label{exa:Steinberg}
For $G$ a semisimple algebraic group over $\FF_q$ and $P$ a $p$-Sylow subgroup, the Steinberg representation of $G$ restricts to a free $\FF_q[P]$-module and thus has totally trivial $G$-cohomology.
\end{example}

Here are some examples to show that the HS property does not always hold. We start with a minimal example.
\begin{example}[Naumann] \label{exa:no HS1}
Put $G = S_3$, let $H$ be the subgroup generated by a transposition, and take $M = \FF_3$ with the action of $G$ being given by the sign character. It is apparent that $M$ has vanishing $H$-cohomology. 
On the other hand, the groups
$H^i(A_3, M)$ are all $\FF_3$-vector spaces and are hence $H$-acyclic,
so \eqref{eq:hs spectral sequence} yields $H^1(S_3, M) = H^1(A_3, M) = \FF_3$.
Explicitly, a nonzero class is represented by the crossed homomorphism
taking one element of order 3 to $+1$ and the other to $-1$, mapping the other elements to 0.
\end{example}

A similar example exists in any odd characteristic $p$ using the dihedral group of order $2p$. For an example in characteristic 2, we offer the following.
\begin{example}[Serre] \label{exa:no HS}
Let $M'$ be a 5-dimensional vector space over $\FF_2$ equipped with a nondegenerate quadratic form $q$. The associated bilinear form $b$ has rank 4; let $K$ be its kernel and put $M = M'/K$. The action of $G = \SO(M',q)$ ($\cong S_6$) preserves $K$ and the induced action on $M$ defines an isomorphism $\SO(M',q) \cong \Sp(M, b) \cong \Sp_4(\FF_2)$. The exact sequence
\[
0 \to K \to M' \to M \to 0
\]
of $G$-modules does not split, so $H^1(G, M)$ is nonzero.

Now split $M$ as a direct sum $M_1 \oplus M_2$ of nonisotropic subspaces and put
$H_i = \SL(M_i)$ and $H = H_1 \times H_2$ ($\cong S_3 \times S_3$). As in Example~\ref{exa:Steinberg}, $M_1$ has no nonzero $H_1$-invariants and restricts to a free module over $\FF_2[P_1]$ for $P_1$ a $2$-Sylow subgroup of $H_1$; it follows that $M_1$ has totally trivial $H_1$-cohomology, hence also totally trivial $H$-cohomology by \eqref{eq:hs spectral sequence}.
Similarly, $M_2$ has totally trivial $H$-cohomology, as then does $M$. We conclude that the inclusion $H \subseteq G$ does not have the HS property.
\end{example}

\section{Analytic group actions}
\label{sec:analytic}

We now introduce the class of group actions to which our main result applies. The basic setup is taken from the work of Lazard \cite{lazard-groups}.

\begin{hypothesis}
Throughout \S\ref{sec:analytic}, let $\Gamma$ be a \emph{profinite $p$-analytic group} in the sense of \cite[III.3.2.2]{lazard-groups}. For example, we may take $\Gamma$ to be a compact $p$-adic Lie group.
\end{hypothesis}

\begin{defn}
For $M$ a $\Gamma$-module, let $C^{\cdot}(\Gamma,M)$ be the complex of inhomogeneous cochains, so that $C^n(\Gamma,M) = \Map(\Gamma^n, M)$ and for $h \in C^n(\Gamma,M)$ and $\gamma_0,\dots,\gamma_n \in \Gamma$,
\begin{align*}
(dh)(\gamma_0, \dots, \gamma_n) &= \gamma_0 h(\gamma_1,\dots,\gamma_n) \\
&+ \sum_{i=1}^n (-1)^i h(\gamma_0,\dots, \gamma_{i-2}, \gamma_{i-1} \gamma_i, \gamma_{i+1},\dots,\gamma_n) \\
&+ (-1)^{n+1} h(\gamma_0,\dots,\gamma_{n-1}).
\end{align*}
For $M$ a topological $\Gamma$-module, let $C^\cdot_{\cont}(\Gamma,M)$ be the subcomplex
of $C^\cdot(\Gamma, M)$ consisting of continuous cochains,
so that $C^n_{\cont}(\Gamma,M) = \Cont(\Gamma^n, M)$.
Let $H^{\cdot}_{\cont}(\Gamma,M)$ be the cohomology groups of $C^{\cdot}_{\cont}(\Gamma, M)$, topologized as subquotients for the compact-open topology; for a more intrinsic interpretation of these groups, see \cite[Proposition~9.4]{flach}. 
\end{defn}

For normal subgroups of $\Gamma$, we again have a Hochschild-Serre spectral sequence.
\begin{lemma} \label{L:hochschild-serre}
For any closed normal subgroup $\Gamma'$ of $\Gamma$ and any topological $\Gamma$-module $M$, there is a spectral sequence
\[
E_2^{p,q} = H^p_{\cont}(\Gamma/\Gamma', H^q_{\cont}(\Gamma', 
M))
\Longrightarrow H^{p+q}_{\cont}(\Gamma, M).
\]
\end{lemma}
For our purposes, convergence of the spectral sequence may be interpreted at the level of bare abelian groups, but it also makes sense at the level of topological groups: starting from $E_2$, each stage of the spectral sequence induces a subquotient topology on the subsequent stage, and $H^{p+q}_{\cont}(\Gamma, M)$ admits a filtration by subgroups (not guaranteed to be closed) whose subquotients are homeomorphic to the corresponding terms of $E_\infty$.
\begin{proof}
Since $\Gamma$ and $\Gamma'$ are profinite, the surjection of topological spaces $\Gamma \to \Gamma/\Gamma'$ admits a continuous section. Consequently, the explicit construction of the spectral sequence for finite groups given in \cite[\S 2]{hochschild-serre} carries over without change.
For further discussion, see \cite[\S V.3.2]{lazard-groups}.
\end{proof}

\begin{defn}
Let $A$ be the completion of the group ring $\Zp[\Gamma]$ with respect to the $p$-augmentation ideal $\ker(\Zp[\Gamma] \to \Fp)$. Put $I = \ker(A \to \Fp)$;
we view $A$ as a filtered ring using the $I$-adic filtration.
We also define the associated valuation:
for $x \in A$, let $w(A;x)$ be the supremum of those nonnegative integers $i$ for which $x \in I^i$.
\end{defn}

\begin{defn}
An \emph{analytic $\Gamma$-module} is a left $A$-module $M$ complete with respect to a valuation $w(M; \bullet)$ for which there exist $a>0, c \in \RR$ such that
\[
w(M; xy) \geq a w(A; x) + w(M; y) + c \qquad (x \in A, y \in M).
\]
Equivalently, there exist 
an open subgroup $\Gamma_0$ of $\Gamma$ and a constant $a > 0$ such that
\[
w(M; (\gamma-1)y) \geq w(M; y) + a \qquad (\gamma \in \Gamma_0, y \in M).
\]
\end{defn}

\begin{example} \label{exa:lazard finite}
Let $M$ be a torsion-free $\Zp$-module of finite rank on which $\Gamma$ acts continuously. Then $M$ is an analytic $A$-module for the valuation defined by any basis; see \cite[Proposition~V.2.3.6.1]{lazard-groups}.
\end{example}

\begin{defn}
Let $M$ be a continuous $\Gamma$-module. A cochain $\Gamma^i \to M$ is \emph{analytic} if
for every homeomorphism between an open subspace $U$ of $\Gamma^i$ and an open subspace $V$ of $\Zp^n$ for some nonnegative integer $n$, the induced function $V \to M$ is locally analytic (i.e., locally represented by a convergent power series expansion).
Let $C^i_{\an}(\Gamma,M) \subseteq C^i_{\cont}(\Gamma,M)$ be the space of analytic cochains. 

Suppose now that $M$ is an analytic $\Gamma$-module. Then by the proof of \cite[Proposition~V.2.3.6.3]{lazard-groups}, $C^i_{\an}(\Gamma,M)$ is a subcomplex of $C^i_{\cont}(\Gamma,M)$; we thus obtain \emph{analytic cohomology} groups $H^i_{\an}(\Gamma,M)$ and natural homomorphisms $H^i_{\an}(\Gamma,M) \to H^i_{\cont}(\Gamma,M)$.
\end{defn}

\begin{theorem}[Lazard] \label{T:lazard analytic}
If $M$ is an analytic $\Gamma$-module, then the inclusion $C^i_{\an}(\Gamma,M) \to C^i_{\cont}(\Gamma,M)$ is a quasi-isomorphism. That is, the continuous cohomology of $M$ can be computed using analytic cochains.
\end{theorem}
\begin{proof}
In the context of Example~\ref{exa:lazard finite}, this is the statement of 
\cite[Th\'eor\`eme~V.2.3.10]{lazard-groups}. However, the proof of this statement only uses the stronger hypothesis in the proof of \cite[Proposition~V.2.3.6.1]{lazard-groups}, which we have built into the definition of an analytic $\Gamma$-module. The remainder of the proof of \cite[Th\'eor\`eme~V.2.3.10]{lazard-groups} thus carries over unchanged.
\end{proof}

\begin{remark}
In considering Theorem~\ref{T:lazard analytic}, it may help to consider the first the case of 1-cocycles: every 1-cocycle is cohomologous to a crossed homomorphism, which is analytic because of how it is determined by its action on topological generators.
\end{remark}

\section{The HS property for some analytic group actions}

We now establish our main result, which gives an analogue of the HS property for certain analytic group actions.

\begin{theorem} \label{T:kill analytic cohomology}
Let $\Gamma$ be a profinite $p$-analytic group.
Let $H$ be a pro-$p$ procyclic subgroup of $\Gamma$ (i.e., it is isomorphic to $\ZZ_p$). 
Let $M$ be a analytic $\Gamma$-module which is a Banach space over some nonarchimedean field of characteristic $p$ with a nontrivial absolute value.
(It is not necessary to require $\Gamma$ to act on this field.)
If $H^i_{\cont}(H, M) = 0$ for all $i \geq 0$,
then $H^i_{\cont}(\Gamma, M) = 0$ for all $i \geq 0$.
\end{theorem}
\begin{proof}
Let $\eta$ be a topological generator of $H$.
The vanishing of $H^0_{\cont}(H,M)$ and $H^1_{\cont}(H,M)$ means that $\eta-1$ is a bijection on $M$;
by the Banach open mapping theorem \cite[\S I.3.3, Th\'eor\`eme~1]{bourbaki-evt}, $\eta-1$ admits a bounded inverse.
Since $M$ is of characteristic $p$, for each nonnegative integer $n$ the actions of $\eta^{p^n}-1$ and $(\eta-1)^{p^n}$ coincide; hence $\eta^{p^n}-1$ also has a bounded inverse.

We next make some reductions. Recall that $M$ has been assumed to be an analytic $\Gamma$-module. We may thus choose a pro-$p$-subgroup $\Gamma_0$ of $\Gamma$ on which the logarithm map defines a bijection with $\Zp^h$ for some $h$, such that for some $c_0 \in (0,1)$ we have
\[
\left| (\gamma-1)y \right| \leq c_0 \left| y \right| \qquad (\gamma \in \Gamma_0, y\in M).
\]
By the previous paragraph,
we may also assume $\eta \in \Gamma_0$. 
Using Lemma~\ref{L:hochschild-serre}, we may also assume $\Gamma = \Gamma_0$.
By Theorem~\ref{T:lazard analytic}, to check that $H^i_{\cont}(\Gamma,M) = 0$ it suffices to check that $H^i_{\an}(\Gamma,M) = 0$.

Let $\Gamma_n$ be the subgroup of $\Gamma_0$ which is the image of $p^n \Zp^h$ under the exponential map. For $c_0$ as above, we have
\begin{equation} \label{eq:exponential convergence}
\left| (\gamma-1)(y) \right| \leq c_0^{p^n} \left| y \right| \qquad
(n \geq 0, \gamma \in \Gamma_n, y \in M).
\end{equation}
For $c \in (0,c_0]$, we say that a cochain $f: \Gamma^n \to M$ is \emph{$c$-analytic}
if there exists $d>0$ such that
\begin{equation} \label{eq:analytic cochain}
\left| f(\gamma_1,\dots,\gamma_n) - f(\gamma_1 \eta_1, \dots,
\gamma_n \eta_n) \right| \leq d c^{p^{i_1 + \cdots + i_n}}
\quad (\gamma_1,\dots,\gamma_n \in \Gamma; i_1,\dots,i_n \geq 0; \eta_j \in \Gamma_{i_j}).
\end{equation}
Using the fact that $M$ is of characteristic $p$, one may check that any analytic cochain in the sense of Lazard is $c$-analytic for some $c>0$.
This means that $C^n_{\an}(\Gamma,M)$ can be written as the union of the subspaces
$C^n_{\an,c}(\Gamma,M)$ of $c$-analytic cochains over all $c \in (0, c_0]$. Moreover,
using \eqref{eq:exponential convergence} we see that $C^n_{\an,c}(\Gamma,M)$ is a subcomplex of $C^n_{\an}(\Gamma,M)$, so to prove the theorem it suffices to check the acyclicity of each $C^n_{\an,c}(\Gamma,M)$.

From now on, fix $c \in (0, c_0]$.
We define a norm on $C^n_{\an,c}(\Gamma,M)$ assigning to each cochain $f$ the minimum $d \geq 0$ for which \eqref{eq:analytic cochain} holds; note that $C^n_{\an,c}(\Gamma,M)$ is complete with respect to this norm.
For $m \geq 0$, we define a chain homotopy $h_m$ on $C^\cdot_{\an,c}(\Gamma,M)$ by the following formula: for $f_n \in C^n_{\an,c}(\Gamma,M)$,
\[
h_m(f_n)(\gamma_1,\dots,\gamma_{n-1})
= (\eta^{p^m}-1)^{-1} \sum_{i=1}^n (-1)^{i-1} 
f_n(\gamma_1,\dots,\gamma_{i-1},
\eta^{p^m}, \gamma_{i}, \dots, \gamma_{n-1}).
\]
We then compute that
\begin{align*}
&(d \circ h_m + h_m \circ d - 1)(f_n)(\gamma_1,\dots,\gamma_n) \\
&=
(\gamma_1 (\eta^{p^m}-1)^{-1} - (\eta^{p^m}-1)^{-1} \gamma_1)
\sum_{i=1}^{n} (-1)^{i-1} f_n(\gamma_2,\dots,\gamma_i,\eta^{p^m},
\gamma_{i+1},\dots,\gamma_n)\\
&-
\sum_{i=1}^n (\eta^{p^m}-1)^{-1}(f_n(\gamma_1,\dots,\gamma_{i-1}, \eta^{p^m} \gamma_i, \gamma_{i+1},\dots,\gamma_n) - f_n(\gamma_1,\dots,\gamma_{i-1},
\gamma_i \eta^{p^m}, \gamma_{i+1}, \dots, \gamma_n)).
\end{align*}
To bound the right side of this equality, write
\[
\gamma(\eta^{p^m}-1)^{-1} - (\eta^{p^m}-1)^{-1}\gamma
=(\eta^{p^m}-1)^{-1}(\eta^{p^m} \gamma)(1 - \gamma^{-1} \eta^{-p^m} \gamma \eta^{p^m})(\eta^{p^m}-1)^{-1}.
\]
Then note that if $\gamma_i \in \Gamma_j$, then $\eta^{p^m} \gamma_i$ and $\gamma_i \eta^{p^m}$ differ by an element of $\Gamma_{m+j+1}$.
Finally, let $t>0$ be the operator norm of the inverse of $\eta-1$ on $M$;
then $\eta^{p^m}-1$ has an inverse of operator norm at most $t^{p^m}$.
Fix $\epsilon \in (0,1)$; for $m$ sufficiently large, we have 
\[
\max\{t^{2p^{m}} c^{p^{2m}},
t^{p^m} c^{p^{m+1}}\} < 1-\epsilon.
\]
For such $m$, the map $d \circ h_m + h_m \circ d - 1$ acts on $C^n_{\an,c}(\Gamma,M)$ with operator norm at most $1-\epsilon$; consequently, there is an invertible map on $C_{\an,c}(\Gamma,M)$ which is homotopic to zero. This proves the claim.
\end{proof}

Note that Example~\ref{exa:no HS1} and Example~\ref{exa:no HS} show that Theorem~\ref{T:kill analytic cohomology}
cannot remain true if we drop the condition that $H$ be pro-$p$. However, it does not resolve the following question.
\begin{question}
Does Theorem~\ref{T:kill analytic cohomology}
 remain true if we drop the condition that $H$ be procyclic?
This does not follow from Theorem~\ref{T:kill analytic cohomology} because the hypothesis of the theorem is not preserved upon replacing $H$ with a subgroup
(Remark~\ref{R:no triviality descent}).
\end{question}

\section{Examples from $p$-adic Hodge theory}
\label{sec:examples}

We conclude with some examples of Theorem~\ref{T:kill analytic cohomology} which are germane to $p$-adic Hodge theory.

\begin{defn}
For any ring $R$ of characteristic $p$, let $\overline{\varphi}: R \to R$ denote the $p$-power Frobenius endomorphism.
\end{defn}

\begin{remark}
We will frequently use the ``Leibniz rule'' for group actions, in the form of the identity
\begin{equation}  \label{eq:Leibniz rule}
(\gamma-1)(\overline{x} \overline{y})
= (\gamma-1)(\overline{x}) \overline{y} + \gamma(\overline{x})
(\gamma-1)(\overline{y}).
\end{equation}
For instance, this holds if $\gamma$ acts on a ring containing $\overline{x}$ and $\overline{y}$, or if it acts compatibly on a ring containing $\overline{x}$ and a module containing $\overline{y}$ (or vice versa).
\end{remark}

\begin{prop} \label{P:affinoid action analytic}
Let $F$ be a complete discretely valued field of characteristic $p$.
Let $R$ be an affinoid algebra over $F$.
Let $M$ be a finitely generated $R$-module.
Let $\Gamma$ be a profinite $p$-analytic group acting compatibly on $F,R,M$,
and suppose that there is an open subgroup of $\Gamma$ fixing the residue field of $F$.
Then $M$ is an analytic $\Gamma$-module.
\end{prop}
\begin{proof}
Let $\frako_F$ be the valuation subring of $F$.
Let $\overline{\pi}$ be a uniformizer of $\frako_F$.
By hypothesis, there exists an open subgroup $\Gamma_0$ on $\Gamma$ fixing
$\frako_F/(\overline{\pi})$. Then for any $\gamma \in \Gamma_0$ and any positive integer $n$,
$\gamma^{p^n}$ fixes $\frako_F/(\overline{\pi}^{n+1})$, so $F$ itself is an analytic $\Gamma$-module.

By definition, $R$ is a quotient of the Tate algebra $F\{T_1,\dots,T_n\}$ for some nonnegative integer $n$. Equip $R$ with the quotient norm for some such presentation.
Let $r_i \in R$ be the image of $T_i$.
Since the action of $\Gamma$ on $R$ is continuous, 
for any $c>0$ there exists an open subgroup $\Gamma_0$ of $\Gamma$ such that
\[
\left|(\gamma-1)(f)\right| \leq \frac{c}{2} \left|f \right|, \qquad \left|(\gamma-1)(r_i) \right| \leq \frac{c}{2}
\]
for all $\gamma \in \Gamma_0$, $i \in \{1,\dots,n\}$, $f \in F$.
Then for any $x \in R$, we can lift it to some $y = \sum_{i_1,\dots,i_n=0}^\infty y_{i_1,\dots,i_n} T_1^{i_1} \cdots T_n^{i_n} \in F\{T_1,\dots,T_n\}$
with $|y| \leq 2|x|$, and then observe that 
\begin{align*}
\left| (\gamma-1)(x) \right| &\leq \max\{\left| (\gamma-1)(y_{i_1,\dots,i_n} r_1^{i_1}\cdots r_n^{i_n}) \right|: i_1,\dots,i_n \geq 0 \} \\
&\leq \frac{c}{2} \max\{\left| y_{i_1,\dots,i_n} \right|: i_1,\dots,i_n \geq 0 \} \qquad \mbox{(by \eqref{eq:Leibniz rule})} \\
&= (c/2)\left|y\right| \leq c\left|x\right|.
\end{align*}
It follows that the action of $\Gamma$ on $R$ is analytic.

Since $R$ is noetherian, $M$ may be viewed as a finite Banach module over $R$
by \cite[Proposition~3.7.3/3, Proposition~6.1.1/3]{bgr}.
By choosing topological generators for $M$ as an $R$-module, we may repeat the argument of the previous paragraph to deduce that the action of $\Gamma$ on $M$ is analytic.
\end{proof}

\begin{example} \label{exa:cyclotomic1}
The action of $\Gamma = \Zp^\times$ on $F = \Fp((\overline{\pi}))$
via the substitution $\pi \mapsto (1+\pi)^{\gamma}-1$ is analytic.
By contrast, the induced action on the completion of the perfect closure of $F$ is continuous but not analytic.

Now take $R = F$ and $M = \overline{\varphi}^{-1}(R)/R$. 
By Proposition~\ref{P:affinoid action analytic},
the action of $\Gamma$ on $M$ is analytic.

Put $\gamma = 1 + p^2 \in \Gamma$; this element generates the pro-$p$ procyclic subgroup $H = 1 + p^2 \ZZ_p$ of $\Gamma$. As an $H$-module,
$M$ splits as a direct sum $\bigoplus_{j=1}^{p-1} (1 + \overline{\pi})^{j/p} F$. 
Choose $j \in \{1,\dots,p-1\}$ and put $\overline{y} = (1+\overline{\pi})^{j/p}$.
We have
\[
(\gamma-1)(\overline{\pi}) = (\gamma-1)(1 + \overline{\pi}) = 
((1 + \overline{\pi})^{p^2} -1)(1 + \overline{\pi}).
\]
Thus on one hand,
\[
\left|(\gamma-1)(\overline{x}) \right|
\leq \left| \overline{\pi} \right|^{p^{2}}
\left| \overline{x} \right|
\qquad (\overline{x} \in F);
\]
on the other hand,
\[
\left|(\gamma-1)(\overline{y}) \right|
= \left| \overline{\pi} \right|^{p} \left| \overline{y}\right|,
\]
and by \eqref{eq:Leibniz rule}, we see that for all $\overline{z} \in \overline{y} F$ we have
\[
\left| (\gamma-1)(\overline{z}) \right|
= \left| \overline{\pi} \right|^{p} \left| \overline{z} \right|.
\]
In particular, $\gamma-1$ is bijective on $\overline{y} F$ for each $j$,
so $H^i_{\cont}(H, M) = 0$ for all $i \geq 0$. In this example, $H$ is normal in $\Gamma$,
so we may invoke Lemma~\ref{L:hochschild-serre} to deduce that $H^i_{\cont}(\Gamma,M) = 0$ for all $i \geq 0$.
This calculation plays an essential role in the proof of the Cherbonnier-Colmez theorem
described in \cite{kedlaya-new-phigamma}.
\end{example}

This example generalizes as follows.
\begin{example} \label{exa:cyclotomic2}
Put $F = \FF_p((\overline{\pi}))$ and $R = F\{\overline{t}_1,\dots,\overline{t}_d\}$ for some $d \geq 0$.
The ring $R$ admits a continuous action of $\Gamma = \ZZ_p^\times \rhd \ZZ_p^d$
in which $\gamma \in \ZZ_p^\times$ acts as in Example~\ref{exa:cyclotomic1}
fixing $\ZZ_p^d$, while for $j=1,\dots,d$ an element $\gamma_j$ in the $j$-th copy of $\ZZ_p$ sends $\overline{t}_j$ to $(1 + \overline{\pi})^{\gamma_j} \overline{t}_j$ and fixes $\overline{\pi}$
and $\overline{t}_k$ for $k \neq j$.
Put
$M = \overline{\varphi}^{-1}(R)/R$.
By Proposition~\ref{P:affinoid action analytic}, the actions of $\Gamma$ on $F, R, M$ are analytic.

Put $\Gamma_0 = (1 + p^2 \ZZ_p) \rhd p\ZZ_p^d$. We then have a decomposition
\begin{equation} \label{eq:ab-decomposition}
M \cong 
\bigoplus
(1 + \overline{\pi})^{e_0/p} \overline{t}_1^{e_1/p} \cdots \overline{t}_d^{e_d/p}
R
\end{equation}
of $R$-modules and $\Gamma_0$-modules, in which $(e_0,\dots,e_d)$ runs over
$\{0,\dots,p-1\}^{d+1} \setminus \{(0,\dots,0)\}$.

Choose a tuple $(e_0,\dots,e_d) \neq (0,\dots,0)$
and put $\overline{y} = (1 + \overline{\pi})^{e_0/p} \overline{t}_1^{e_1/p} \cdots \overline{t}_d^{e_d/p}$. Suppose first that $e_j \neq 0$ for some $j>0$.
Let $\gamma$ be the canonical generator of the $j$-th copy of $p \ZZ_p^d$.
Then
\[
\left|(\gamma-1)(\overline{y}) \right| = \left| \overline{\pi} \right| \overline{y};
\]
on the other hand,
\[
\left| (\gamma-1)(\overline{x}) \right|
\leq \left| \overline{\pi} \right|^p \left| \overline{x} \right| \qquad (\overline{x} \in R),
\]
so using \eqref{eq:Leibniz rule} again we see that $\gamma-1$ acts invertibly on $\overline{y} R$. By  Lemma~\ref{L:hochschild-serre}
we have $H^i_{\cont}(\Gamma_0, \overline{y} R) = 0$ for all $i \geq 0$.

Suppose next that $e_0 \neq 0$ but $e_1 = \cdots = e_d = 0$. 
Put $\gamma = 1 +p^2 \in \ZZ_p^\times$.
As in Example~\ref{exa:cyclotomic1}, we see that $\gamma-1$ acts invertibly on $\overline{y} R$. Since $\Zp^\times$ is not normal in $\Gamma$, we must now apply
Theorem~\ref{T:kill analytic cohomology} instead of Lemma~\ref{L:hochschild-serre}
to deduce that $H^i_{\cont}(\Gamma_0, \overline{y} R) = 0$ for all $i \geq 0$.

Putting everything together, we deduce that $H^i_{\cont}(\Gamma_0, M) = 0$ for all $i \geq 0$.
By Lemma~\ref{L:hochschild-serre} once more, we see that
$H^i_{\cont}(\Gamma, M) = 0$ for all $i \geq 0$.
This calculation plays an essential role in a generalization of the Cherbonnier-Colmez theorem described in \cite{kedlaya-liu2}.
\end{example}

\begin{remark}
Another class of examples to be considered in \cite{kedlaya-liu2}, based on Lubin-Tate towers, yields cases in which
$\Gamma = \mathrm{GL}_d(\ZZ_p)$ and the vanishing of cohomology can again be checked using
Theorem~\ref{T:kill analytic cohomology}.
\end{remark}

\section*{Acknowledgments}
Thanks to Niko Naumann and Jean-Pierre Serre for providing Example~\ref{exa:no HS1} and Example~\ref{exa:no HS}, respectively, and to Serre for additional feedback.
Kedlaya was supported by NSF grant DMS-1101343 and UC San Diego
(Stefan E. Warschawski Professorship),
and additionally by NSF grant DMS-0932078 while in residence at MSRI during fall 2014.

\end{document}